\documentclass[11pt]{amsart}
\usepackage[T1]{fontenc}
\usepackage[english]{babel}
\usepackage{amscd,amsmath,amsthm,amssymb,graphics}
\usepackage{lmodern,pst-node}
\usepackage{pstcol,pst-plot,pst-3d}
\usepackage{multicol}
\usepackage{epic,eepic}
\usepackage{amsfonts,amssymb,amscd,amsmath,enumitem,verbatim}
\psset{unit=0.7cm,linewidth=0.8pt,arrowsize=2.5pt 4}

\newpsstyle{fatline}{linewidth=1.5pt}
\newpsstyle{fyp}{fillstyle=solid,fillcolor=verylight}
\definecolor{verylight}{gray}{0.97}
\definecolor{light}{gray}{0.9}
\definecolor{medium}{gray}{0.85}
\definecolor{dark}{gray}{0.6}

\def\frk{\frak}               

\def\Phi{{\frk n}}
\def\Phi{{\frk N}}
\def\opn#1#2{\def#1{\operatorname{#2}}} 
\opn\chara{char} \opn\length{\ell} \opn\pd{pd} \opn\rk{rk}
\opn\projdim{proj\,dim} \opn\injdim{inj\,dim} \opn\rank{rank}
\opn\depth{depth} \opn\grade{grade} \opn\height{height}
\opn\embdim{emb\,dim} \opn\codim{codim}

\opn\Tr{Tr} \opn\bigrank{big\,rank}
\opn\superheight{superheight}\opn\lcm{lcm}
\opn\trdeg{tr\,deg}
\opn\reg{reg} \opn\lreg{lreg} \opn\ini{in} \opn\lpd{lpd}
\opn\size{size}\opn\bigsize{bigsize}
\opn\cosize{cosize}\opn\bigcosize{bigcosize}
\opn\sdepth{sdepth}\opn\sreg{sreg}
\opn\link{link}\opn\fdepth{fdepth}
\opn\deg{deg}
\opn\max{max}
\opn\indeg{indeg}
\opn\min{min}
\opn\psln{psln}
\opn\div{div} \opn\Div{Div} \opn\cl{cl} \opn\Cl{Cl}
\let\epsilon\varepsilon
\let\phi=\varphi
\let\kappa=\varkappa

\opn\Spec{Spec} \opn\Supp{Supp} \opn\supp{supp} \opn\Sing{Sing}
\opn\Ass{Ass} \opn\Min{Min}\opn\Mon{Mon} \opn\dstab{dstab} \opn\astab{astab}
\opn\Syz{Syz}
\opn\Ann{Ann} \opn\Rad{Rad} \opn\Soc{Soc}
\opn\Im{Im} \opn\Ker{Ker} \opn\Coker{Coker} \opn\Am{Am}
\opn\Hom{Hom} \opn\Tor{Tor} \opn\Ext{Ext} \opn\End{End}
\opn\Aut{Aut} \opn\id{id}

\opn\nat{nat}
\opn\pff{pf}
\opn\Pf{Pf} \opn\GL{GL} \opn\SL{SL} \opn\mod{mod} \opn\ord{ord}
\opn\Gin{Gin} \opn\Hilb{Hilb}\opn\sort{sort}
\opn\initial{init}
\opn\ende{end}
\opn\height{height}
\opn\depth{depth}
\opn\type{type}
\opn\ldim{ldim}

\opn\aff{aff} \opn\con{conv} \opn\relint{relint} \opn\st{st}
\opn\lk{lk} \opn\cn{cn} \opn\core{core} \opn\vol{vol}
\opn\link{link} \opn\star{star}\opn\lex{lex}
\opn\gr{gr}

\def\pot#1#2{#1[\kern-0.28ex[#2]\kern-0.28ex]}

\opn\dirlim{\underrightarrow{\lim}}
\opn\inivlim{\underleftarrow{\lim}}
\def\Implies{\ifmmode\Longrightarrow \else
        \unskip${}\Longrightarrow{}$\ignorespaces\fi}
\def\implies{\ifmmode\Rightarrow \else
        \unskip${}\Rightarrow{}$\ignorespaces\fi}
\def\iff{\ifmmode\Longleftrightarrow \else
        \unskip${}\Longleftrightarrow{}$\ignorespaces\fi}

\let\:=\colon
 \theoremstyle{plain}
\newtheorem{Theorem}{Theorem}[section]
 \newtheorem{Lemma}[Theorem]{Lemma}

 \theoremstyle{definition}
 
 \newtheorem{Remark}[Theorem]{Remark}
 
 \newtheorem{Example}[Theorem]{Example}

\let\epsilon\varepsilon
\let\kappa=\varkappa
\def\qed{\ifhmode\textqed\fi
      \ifmmode\ifinner\quad\qedsymbol\else\dispqed\fi\fi}
\def\textqed{\unskip\nobreak\penalty50
       \hskip2em\hbox{}\nobreak\hfil\qedsymbol
       \parfillskip=0pt \finalhyphendemerits=0}
\def\dispqed{\rlap{\qquad\qedsymbol}}

\opn\dis{dis}
\def\pnt{{\raise0.5mm\hbox{\large\bf.}}}

\opn\Lex{Lex}

\begin{document}

\author[Mafi and Naderi]{ Amir Mafi and Dler Naderi}
\title{On the first generalized Hilbert coefficient and depth of associated graded rings}

\address{Amir Mafi, Department of Mathematics, University Of Kurdistan, P.O. Box: 416, Sanandaj, Iran.}
\email{A\_Mafi@ipm.ir}
\address{Dler Naderi, Department of Mathematics, University of Kurdistan, P.O. Box: 416, Sanandaj,
Iran.}
\email{dler.naderi65@gmail.com}

\begin{abstract}
Let $(R,\frak{m})$ be a $d$-dimensional Cohen-Macaulay local ring with infinite residue field. Let $I$ be an ideal of $R$ that has analytic spread $\ell(I)=d$, satisfies the $G_d$ condition, the weak Artin-Nagata property $AN_{d-2}^-$ and $\depth(R/I)\geq\min\lbrace 1,\dim R/I \rbrace$. In this paper, we show that if $j_1(I) = \lambda (I/J) +\lambda [R/(J_{d-1} :_{R} I+(J_{d-2} :_{R}I+I) :_{R}{\frak{m}}^{\infty})]+1$, then  $\depth(G(I))\geq d -1$ and $r_J(I) \leq 2$, where $J$ is a general minimal reduction of $I$. In addition, we extend the result by Sally who has studied the depth of associated graded rings and minimal reductions for an $\frak{m}$-primary ideals.

\end{abstract}


\maketitle
\section*{Introduction}
The associated graded ring of an ideal $I$ of $R$, defined as the graded algebra
\[G(I)= \oplus_{n=0}^{\infty} I^n /I^{n+1}\]
encodes algebraic and geometric properties of $I$. In this paper we investigate the behavior of the depth of the associated graded ring $G(I)$ of an ideal $I$ in a Noetherian local ring $(R,\frak{m})$. Its arithmetical properties, like its depth,
provide useful information, for instance, the high depth of the associated graded ring forces the vanishing of cohomology groups and thereby allows one to compute, or bound, relevant numerical invariants, such as the Castelnuovo-Mumford regularity  (see \cite{JU}, \cite{JK}).
For an $\frak{m}$-primary ideal $I$, the interplay between the Hilbert polynomial of $I$, more precisely its Hilbert coefficients, and
the depth of the associated graded ring has been widely investigated. The classical method, originated from the pioneering work of Sally, studies the interplay between the Hilbert coefficients of an $\frak{m}$-primary ideal and the depth of the associated graded ring.
The classical Hilbert functions are only defined for ideals that are primary to the maximal ideal,
 Achilles and Manaresi in \cite{AM} introduced the concept of the $j$-multiplicity as a generalization of the Hilbert multiplicity and in \cite{AM2} they also defined a generalized  Hilbert function using the bigraded ring of the associated graded ring with respect to the maximal ideal. Flenner, $O^{,}$Carroll and Vogel in \cite{FO} defined the generalized Hilbert function using the $0$-th local cohomology functor. Later Ciuperca in \cite{C} introduced the generalized Hilbert coefficients via the approach of Achilles and Manaresi. Polini and Xie in \cite{PX2} defined the concepts of the generalized Hilbert polynomial and the generalized Hilbert coefficients following the approach of Flenner, $O^{,}$Carroll and Vogel, and they proved that the generalized Hilbert coefficients as defined using the $0$-th local cohomology functor can also be obtained from the generalized Hilbert function of the bigraded ring of the associated graded ring with respect to a suitable ideal.
If $I$ is an $\frak{m}$-primary ideal in a Cohen-Macaulay local ring $R$, the positivity of $e_1(I)$ can be observed from the well-known Northcott's inequality $e_1(I)\geq e_0(I)-\lambda (R/I)=\lambda (R/J)-\lambda (R/I)=\lambda (I/J)$,
where $J$ is a minimal reduction of $I$. When equality holds, the ideal $I$ enjoys nice properties. Indeed, it was shown that $e_1(I)=\lambda (I/J)$ if and only if the reduction number of $I$ is at most $1$, and when this is the case, the associated graded ring of $I$ is Cohen-Macaulay (see \cite{H} and \cite{O}).
Xie in \cite{X} generalized Northcott's inequality to $R$-ideal as follow:
\[j_1(I) \geq \lambda (I/J) + \lambda [R/(J_{d-1} :_{R} I+(J_{d-2} :_{R} I+I) :_{R} {\frak{m}}^{\infty})] \]
and he proved that $j_1(I) =\lambda (I/J)+\lambda [R/(J_{d-1} :_{R} I+(J_{d-2} :_{R} I+I) :_{R} {\frak{m}}^{\infty})]$ if and only if $r(I)\leq 1$ and when this is the case, the associated graded ring of $I$ is Cohen-Macaulay.
Sally in \cite{S} studied $\frak{m}$-primary ideal $I$ satisfying the condition $e_1(I)=\lambda(I/J) +1$, with $e_2(I)> 0$ if $d \geq 2$ and she  proved that if $e_1(I)=\lambda(I/J) +1$, then $\depth(G(I))\geq d-1$ and $r_J(I)\leq 2$ for any minimal reduction $J$ of $I$.
This paper generalizes the above results to ideals of maximal analytic spread. In Section 1, we define generalized Hilbert-Samuel function of $I$ and recall definitions of residual intersections. In Section 2, we prove that if  $(R,\frak{m})$ is a Cohen-Macaulay local ring, $I$ a non $\frak{m}$-primary $R$-ideal which satisfies $\ell (I)=d$, the $G_d$ condition, the $AN_{d-2}^-$ and $\depth(R/I) \geq\min\lbrace 1,\dim R/I \rbrace$, then for a general minimal reduction $J=(x_1, ... ,  x_d)$ of $I$, one has that $j_1(I) \geq \lambda (I/J) +\lambda (I^2/JI)+\lambda [R/(J_{d-1} :_{R} I+(J_{d-2} :_{R} I+I) :_{R} {\frak{m}}^{\infty})]$. When $j_1(I) = \lambda (I/J) +\lambda [R/(J_{d-1} :_{R} I+(J_{d-2} :_{R} I+I) :_{R} {\frak{m}}^{\infty})]+1$ then $\depth(G(I))\geq d-1$ and $r_J(I) \leq 2$ for any general minimal reduction $J=(x_1, ... ,  x_d)$ of $I$.

\section{ Preliminary}
 In this paper, we always assume that $(R,\frak{m}, k)$ is a Noetherian local ring of dimension $d$ with maximal ideal $\frak{m}$ and infinite residue field $k$.
 For an ideal $I$ of $R$, we will denote by $G(I)=\oplus_{n=0}^{\infty} I^n /I^{n+1}$ the associated graded algebra of $I$, and by $\mathcal{F}(I) = \oplus_{n=0}^{\infty} I^n /{\frak{m}}I^{n}$ the fiber cone of $I$. The dimension of $G(I)$ is always $d$ and the the dimension of $\mathcal{F}(I)$ is called the analytic spread of $I$ and is denoted by $\ell(I)$.

 Recall that an ideal $J\subseteq I$ is called a reduction of $I$ if $JI^r=I^{r+1}$ for some non-negative integer $r$. The least such $r$ is denoted by $r_J(I)$. A reduction ideal is minimal if it is minimal with respect to inclusion. The reduction number $r(I)$ of $I$ is defined as $\min\lbrace r_J(I) ~|~ J$ is a minimal reduction of I$\rbrace$. Since $R$ has infinite residue field, the minimal number of generators $\mu (J)$ of any minimal reduction $J$ of $I$ equals the analytic spread $\ell(I)$.

 Let $I=(a_1, ..., a_t)$ and write $x_i=\sum_{j=1}^{t}\lambda_ {ij}a_j$ for $1 \leq i \leq s$ and $\lambda_ {ij} \in R^{st}$. The elements $x_1, ..., x_s$ form a sequence of general elements in $I$ (equivalently $x_1, ..., x_s$ are general in $I$) if there exists a Zariski dense open subset $U$ of $k^{st}$ such that the image $(\overline{\lambda_ {ij}} )\in U$. When $s =1$, $x =x_1$ is said to be general in $I$.
 Furthermore, general $\ell(I)$ elements in $I$ form a minimal reduction $J$ whose $r_J(I)$ coincides with the reduction number $r(I)$ (see \cite{HS} or \cite{T}). One says that $J$ is a general minimal reduction of $I$ if it is generated by $\ell(I)$ general elements in $I$.

We recall the concept of the generalized Hilbert-Samuel function of $I$. In $G(I)= \oplus_{n=0}^{\infty} I^n/I^{n+1}$  the associated graded ring of $I$, as the homogeneous components of $G(I)$ may not have finite length, one considers the submodule of $G(I)$ of elements supported on $\frak{m}$ as follow:
 \[ W=\lbrace \xi \in G ~| ~\exists ~ t >0 ~such ~that ~\xi.{\frak{m}}^t=0\rbrace =H^{0}_{\frak{m}}(G) =\oplus_{n=0}^{\infty} H^{0}_{\frak{m}}( I^n/I^{n+1}).\]
Since $W$ is a finite graded module over $G(I) \otimes_R R/{\frak{m}}^\alpha $ for some $\alpha \geq 0$, its Hilbert-Samuel function $H_W(n) =\sum_{j=0}^{n} \lambda (H^{0}_{m}(I^j/I^{j+1}))$ is well defined. The generalized Hilbert-Samuel function of $I$ is defined to be: $H_I(n) =H_W(n)$ for every $n \geq 0$. The definition of generalized Hilbert-Samuel function was introduced by Flenner, $O^{,}$Carroll and Vogel in \cite[ Definition 6.1.5]{FO}, and studied later by Polini and Xie in \cite{PX2}. Since $\dim_G W \leq\dim R=d$, $H_I(n)$ is eventually a polynomial of degree at most $d$,
\[ P_I (n) =\sum_{i=0}^{d}(-1)^{i} j_{i}(I) \left( {\begin{array}{*{20}{c}}
{n + d - i}\\
{d - i}
\end{array}} \right).\]
Polini and Xie defined $P_I(n)$ to be the generalized Hilbert-Samuel polynomial of $I$ and $j_i(I)$, $0 \leq i \leq d$, the generalized Hilbert coefficients of $I$. The normalized leading coefficient $j_0(I)$ is called the $j$-multiplicity of $I$ (see \cite{AM}, or \cite{PX2}). The next normalized coefficient $j_1(I)$ is called the generalized first Hilbert coefficient. Xie provided a formula relating the length $\lambda (I^{n+1}/JI^n)$ to the difference $P_I(n) -H_I(n)$, where $I$ is an  $R$-ideal (genralized of fundamental Lemma of Huneke \cite[Lemma 2.4]{H} and extension of this Lemma by Huckaba \cite[Theorem 2.4]{HUC}) and find a formula for generalized first Hilbert coefficient (see \cite[Theorem 3.2, Corollary 3.3]{X}).
In general, $\dim_{G} W \leq \ell(I) \leq d$ and equalities hold if and only if $\ell(I) =d$. Therefore $j_0(I) \ne 0$ if and only if $\ell(I)=d$ (see \cite{AM2} or \cite{NU}).
If $I$ is an $\frak{m}$-primary, each homogeneous component of $G$ has finite length, thus $W=G$ and the generalized Hilbert-Samuel function coincides with the usual Hilbert-Samuel function; in particular, the generalized Hilbert coefficients $j_i(I)$, $0 \leq i\leq d$, coincide with the usual Hilbert coefficients $e_i(I)$.
We now recall some definitions and facts from the theory of residual intersections which will be used frequently in the rest of the paper.
\begin{enumerate}

\item[(i)] The ideal $I$ is said to satisfy the $G_{s+1}$ condition if for every ${\frak{p}}\in V(I)$ with $\height{\frak{p}}=i\leq s$, the ideal $I_{\frak{p}}$ is generated by $i$ elements, i.e., $I_{\frak{p}}=(x_1, ..., x_i)_{\frak{p}}$ for some $x_1, ..., x_i$ in $I$.
\item[(ii)]
Let $J_s=(x_1, ..., x_s)$, where $x_1, ..., x_s$ are elements in $I$. Then $J_{s}:I$ is called a $s$-residual intersection of $I$ if $I_{\frak{p}}=(x_1, ..., x_s)_{\frak{p}}$ for every ${\frak{p}}\in\Spec(R)$ with $\dim R_{\frak{p}}\leq s -1$.
\item[(iii)]
A $s$-residual intersection $J_{s}:I$ is called a geometric $s$-residual intersection of $I$ if, in addition, $I_{\frak{p}}=(x_1, ..., x_s)_{\frak{p}}$ for every ${\frak{p}}\in V(I)$ with $\dim R_{\frak{p}}\leq s $.
\item[(iv)]
If $I$ satisfies the $G_s$ condition, then for general elements $x_1, ..., x_s$ in $I$ and each $0\leq i < s$, the ideal $J_{i}:I$ is a geometric $i$-residual intersection of $I$, and $J_{s}:I$ is a $s$-residual intersection of $I$ (see \cite{U} and \cite{PX}).
\item[(v)]
Let $R$ be Cohen-Macaulay, the ideal $I$ has the Artin-Nagata property $AN_{s}^{-}$ if, for every $0 \leq i\leq s$ and every geometric $i$-residual intersection $J_{i}:I$ of $I$, one has that $R/J_{i}:I$ is Cohen-Macaulay (see \cite{U}).
\end{enumerate}

It is well-known that the $G_d$ condition and the Artin-Nagata property $AN_{d-2}^{-}$ are automatically satisfied by any $\frak{m}$-primary ideal in a Cohen-Macaulay local ring.
From now on, we will assume $I$ has $\ell(I) =d$ and satisfies the $G_d$ condition. Let $J=(x_1, ..., x_d)$, where $x_1, ..., x_d$ are general elements in $I$, i.e., $J$ is a general minimal reduction of $I$. For $i \leq d -1$, set $J_i=(x_1, ..., x_i)$ (with the convention $J_i=(0)$ if $i \leq 0$), $R^{d-1}=R/J_{d-1}:I^{\infty}$, where $J_{d-1}:I^{\infty}= \lbrace a \in R~ |~ \exists ~ t >0 ~such~that~a .I^t \subseteq  J_{d-1}\rbrace$. Then $R^{d-1}$ is a $1$-dimensional Cohen-Macaulay local ring and $IR^{d-1}$ is ${\frak{m}}R^{d-1}$-primary. Hence the generalized Hilbert-Samuel function $H_{IR^{d-1}}(n)$ and the generalized Hilbert-Samuel polynomial $P_{IR^{d-1}}(n)$ are respectively the usual Hilbert-Samuel function and the usual Hilbert-Samuel polynomial of $IR^{d-1}$. Note that $H_{IR^{d-1}}(n)$ and hence $P_{IR^{d-1}}(n)$ do not depend on choices of general elements $x_1, ..., x_{d-1}$ in $I$, and $P_{IR^{d-1}}(n) =e_0(IR^{d-1})(n +1) -e_1(IR^{d-1})$, where $e_0(IR^{d-1}) =\lambda (R^{d-1}/x_d R^{d-1}) =j_0(IR^{d-1})$(see \cite{PX2}). If $R$ is Cohen-Macaulay and $I$ is $\frak{m}$-primary, then $e_1(IR^{d-1}) =e_1(I)$ (see for instance \cite[Proposition 1.2]{RV}). But they are in general not the same.

\section{ main result}
In this section we find lower bound for first Hilbert coefficient $j_1(I)$, where $I$ is not an $\frak{m}$-primary ideal with $\ell(I) =d$, and satisfies the $G_d$ condition and the $AN_{d-2}^{-}$, $J$ is a general minimal reduction of $I$ (see Theorem \ref{P2}). Also, we prove that if  $j_1(I) = \lambda (I/J) +\lambda [R/(J_{d-1} :_{R} I+(J_{d-2} :_{R} I+I) :_{R}{\frak{m}}^{\infty})] +1$, then $G(I)$ is almost Cohen-Macaulay (see Theorem \ref{P3}).

\begin{Remark} \label{R1}
Assume $R$ is Cohen-Macaulay. Let $I$ be an $R$-ideal which satisfies $\ell (I) =d$, the $G_d$ condition, the $AN_{d-2}^-$ and $J=(x_1, ... ,  x_d)$ general minimal reduction of $I$. Since $I$ satisfies $\ell (I) =d$ and the $G_d$ condition, one has that $J_i:I$ is a geometric $i$-residual intersection of $I$, where $J_i=(x_1, ..., x_i)$, $0 \leq i \leq d -1$ \cite[Lemma 3.1]{PX}. By the weak Artin-Nagata property $AN_{d-2}^-$, one has that $J_i:I^{\infty}=J_i:I=J_i:x_{i+1}$ and $(J_i:I^{\infty}) \cap I=(J_i:I) \cap I=J_i$ for $0 \leq i \leq d -1$ \cite[Lemma 3.2]{PX} and \cite{U}. Therefore  $I^{n+1} \cap (J_i:I^{\infty})  =I^{n+1} \cap  J_i $ and $JI^n \cap (J_i:I^{\infty})  =JI^n \cap J_i $ for $n \geq 0$ and $0 \leq i \leq d -1$.

\end{Remark}

\begin{Lemma}\label{l1}
Assume $R$ is Cohen-Macaulay. Let $I$ be a non $\frak{m}$-primery  $R$-ideal which satisfies $\ell (I) =d$, the $G_d$ condition, the $AN_{d-2}^-$, $J=(x_1, ... ,  x_d)$ general minimal reduction of $I$ and $\depth(R/I) \geq\min\lbrace 1,\dim R/I \rbrace$. Then $I^2 \cap J_{i} =J_{i} I$ for any $0 \leq i \leq d-1$
\end{Lemma}

\begin{proof}
Since $\depth(R/I)\geq 1$, by  \cite[Lemma 3.2(f)]{PX} we have $J_{d-1} :_{I^2}  I^{\infty}= J_{d-1} I$. Then by Remark \ref{R1}, $I^2 \cap J_{d-1} = J_{d-1}I$. Now we show that $I^2 \cap J_{d-2} = J_{d-2}I$,
\[I^2 \cap J_{d-2} \subseteq  I^2 \cap J_{d-1} = J_{d-1}I =J_{d-2}I +x_{d-1} I\]
 intersect both side with $J_{d-2}$ we have
\begin{align*}
I^2 \cap J_{d-2} &=(J_{d-2}I +x_{d-1} I) \cap J_{d-2}\\
&=J_{d-2}I +(J_{d-2} \cap x_{d-1} I)\\
&=J_{d-2}I +(J_{d-2} \cap x_{d-1}R \cap x_{d-1} I)\\
&=J_{d-2}I+x_{d-1}(J_{d-2}: x_{d-1} \cap I)\\
&=J_{d-2}I+x_{d-1}J_{d-2}\\
&= J_{d-2}I
\end{align*}
By the similar argument we have
 $I^2 \cap J_{i} =J_{i} I$ for any $0 \leq i \leq d-3$ .
\end{proof}

\begin{Lemma}\label{l2}
Assume $R$ is Cohen-Macaulay. Let $I$ be an $R$-ideal which satisfies $\ell (I) =d \geq 3$, the $G_d$ condition, the $AN_{d-2}^-$, $J=(x_1, ... ,  x_d)$ general minimal reduction of $I$ and $\depth(R/I) \geq \min\lbrace 1,\dim R/I \rbrace$. Then $\depth(R^{1} /IR^{1} ) \geq\min\lbrace 1,\dim (R^{1} /IR^{1})  \rbrace$, where $R^{1}=R/J_{1} :I$ and $IR^{1}$ is the image of $I$ in the quotient ring $R^{1}$.
\end{Lemma}
\begin{proof}
If $I$ is an $\frak{m}$-primary ideal the assertion is clear, so we assume that $\dim R/I \geq 1$.
Set $R^{0} = R/0 : I$. Then $R^{0}$ is Cohen-Macaulay since $I$ satisfies ${AN_{d-2}^{-}}$. Note that $\dim R^{0}=\dim R = d$, $\grade (IR^{0}) \geq 1$, $IR^{0}$ still satisfies $G_d$ condition, ${AN_{d-2}^{-}}$ , $\ell(IR^{0}) = \ell(I) = d$ (see for instance \cite{U}). Hence $\depth(R^{0}/x_{1}R^{0})=d-1$  or $\depth(R/J_1+0:I)=d-1$.
By the exact sequence
\[0 \longrightarrow J_1:I/J_1 +0:I \longrightarrow R/J_1 +0:I \longrightarrow R/J_1:I \longrightarrow 0\]
one has $\depth (J_1:I/J_1 +0:I) \geq d-1$.
Since
\[ J_1: I +I/ 0:I+I =J_1: I +I+0:I/ 0:I+I \cong J_1 : I / (0:I +I) \cap J_1 : I =J_1 :I /0:I +J_1\]
we have $\depth (J_1: I +I/ 0:I+I)\geq d-1.$
Hence by the exact sequence
\[ 0  \longrightarrow  0 : I \longrightarrow R/I \longrightarrow R/(0 : I + I) \longrightarrow 0,\]
and since $\depth(0 :I) =d$ (see \cite{U}), we have $\depth (R/(0 : I + I)) \geq 1$.
Now by the following exact sequence
\[ 0 \longrightarrow J_1: I +I/ 0:I+I \longrightarrow R/0:I+I \longrightarrow R/J _1: I +I \longrightarrow 0\]
we have $\depth (R/J _1: I +I) \geq 1$. Therefore $\depth(R^{1} /IR^{1} ) \geq\min\lbrace 1,\dim (R^{1} /IR^{1}) \rbrace$.
\end{proof}

\begin{Remark}
Assume $R$ is Cohen-Macaulay of dimension $d$. Let $I$ be an $\frak{m}$-primary ideal of $R$ and $I^2 \cap J=JI$, for a minimal reduction $J$ of $I$.  By formula of Huckaba and Marley \cite[4.2]{HM}, we have $e_1(I) \geq  \lambda (I/J) +\lambda (I^2 / JI)$, and if  $e_1(I) = \lambda (I/J) +1$ then $\depth(G(I)) \geq d-1$.
\end{Remark}

\begin{Theorem}\label{p1}
Assume that $R$ is Cohen-Macaulay. Let $I$ be  a non $\frak{m}$-primary  $R$-ideal which satisfies $\ell (I)=2$, the $G_2$ condition, the $AN_{0}^-$ and $\depth(R/I)\geq \min\lbrace 1, \dim R/I \rbrace$. Then for a general minimal reduction $J=(x_1, x_2)$ of $I$, one has that $j_1(I) \geq \lambda (I/J) +\lambda (I^2/JI) + \lambda [R/(J_{1} :_{R} I+(0 :_{R} I+I) :_{R} {\frak{m}}^{\infty})]$.
\end{Theorem}
\begin{proof}
By \cite[Corollary 3.2]{X} and Remark \ref{R1}, we have
\begin{align*}
 j_1(I) &= \sum\limits_{n = 0}^\infty  \lambda({{{I^{n + 1}} }}/{{J{I^n} }}) - \sum\limits_{n = 0}^\infty  \lambda({{{I^{n + 1}} \cap J_{1}}}/{{J{I^n} \cap J_{1}}})\\
&+  \lambda (R/(I+J_{1}: I))- \lambda ({H_{\frak{m}}^0}(R/I+ 0 : I)) \\
& +\sum\limits_{n = 1}^\infty [ \lambda({ \widetilde  {L_n^{0}} })-\lambda({{L_n^{0}} })+ \lambda({{N_n^{0}} })]
\end{align*}
where
\[ \widetilde{ L_{n}^{0}}=I^n \cap J_1 / [I^{n+1} \cap J_1 +x_1 I^{n-1}]\]
\[L_{n}^{0}=I^{n+1}:_{I^n \cap J_1}{\frak{m}}^{\infty}  /[I^{n+1} \cap J_1 +x_1(I^n:_{I^{n-1}}{\frak{m}}^{\infty} )]\]
\[N_{n}^{0}= (I^n \cap J_1 +I^{n+1}):_{I^n} {\frak{m}}^{\infty} / [I^n \cap J_1 + (I^{n+1} :_{I^n} {\frak{m}}^{\infty} )]. \]
Assume $R^{1}= R/(J_{1} : I^{\infty})$ and $IR^{1}=(I+J_{1} : I^{\infty})/(J_{1}: I^{\infty})$. By the fact that $\lambda({{{I^{n + 1}} }}/{{J{I^n} }})<\infty $, $\lambda({{{I^{n + 1}}\cap J_{1}}}/{{J{I^n} \cap J_{1}}})=\lambda({{{I^{n + 1}}}} \cap (JI^{n} +J_{1} : I^{\infty})/{{J{I^n}}})<\infty$ (see \cite[Lemma 4.6]{PX}) one has
\begin{align*}
\lambda ({I}^{n+1}R^{1}/ JI^{n}R^{1}) &=\lambda ((I^{n+1}+J_{1} : I^{\infty})/(JI^{n} +J_{1} : I^{\infty}))\\
&= \lambda (I^{n+1} /JI^{n} +(I^{n+1} \cap J_{1} : I^{\infty}))\\
&= \lambda({{{I^{n + 1}} }}/{{J{I^n} }})- \lambda({{{I^{n + 1}} \cap J_{1}}}/{{J{I^n} \cap J_{1}}})
\end{align*}
Since $\dim(R/I) \geq 1$ ( $I$ is a non $\frak{m}$-primary ideal) by hypothesis $\depth(R/I) \geq 1$, then by \cite[Lemma 3.2(f)]{PX} $I^2 \cap J_{1} : I^{\infty}= J_1 I$. Now by summing these equations over all $n \geq 1$ and using the fact that $\lambda (IR^{1}/JR^{1})=\lambda (I/J)$, $\lambda ({I}^{2}R^{1}/JIR^{1})=\lambda (I^2/(JI+ (I^2 \cap J_{1} : I^{\infty})))=\lambda(I^2/JI)$, we see that
\begin{align*}
\sum\limits_{n = 0}^\infty  \lambda({{{I^{n + 1}} }}/{{J{I^n} }}) - \sum\limits_{n = 0}^\infty  \lambda({{{I^{n + 1}} \cap J_{1}}}/{{J{I^n} \cap J_{1}}}) = \sum\limits_{n = 0}^\infty \lambda ({I}^{n+1}R^{1}/ JI^{n}R^{1})\\
=\lambda (I/J) +\lambda (I^2/JI) +  \sum\limits_{n = 2}^\infty \lambda (I^{n+1}/(JI^{n} +(I^{n+1} \cap J_{1} : I^{\infty}))).
\end{align*}
Claim 1:
\[\lambda (R/(I+J_{1}: I))- \lambda ({H_{\frak{m}}^0}(R/I+ 0 : I)) \geq  \lambda [R/(J_{1} :_{R} I+(0 :_{R} I+I) :_{R} {\frak{m}}^{\infty})].\]
 Proof of Claim 1: We can write
\begin{align*}
\lambda (R/(I+J_{1}: I))- \lambda ({H_{\frak{m}}^0}(R/I+ 0 : I)) &= \lambda (R/(I+J_{1}: I))- \lambda ({H_{\frak{m}}^0}(R/I)) \\
& + \lambda ({H_{\frak{m}}^0}(R/I))- \lambda ({H_{\frak{m}}^0}(R/I+ 0 : I))
\end{align*}
by applying  the long exact sequence of local cohomology on the following exact sequence
\[0\longrightarrow  0:I\longrightarrow R/I\longrightarrow R/(0 : I+I) \longrightarrow 0,\]
and since $\depth(0 :I) =2$ (see \cite{U}), we have
\[\lambda ({H_{\frak{m}}^0}(R/I))- \lambda ({H_{\frak{m}}^0}(R/I+ 0 : I)) \geq 0.\]
To complete the proof of Claim 1, we need to show that
\[\lambda (R/(I+J_{1}: I))- \lambda ({H_{\frak{m}}^0}(R/I)) = \lambda [R/(J_{1} :_{R} I+(0 :_{R} I+I) :_{R}{\frak{m}}^{\infty})]. \]
 One has\\
$
  \lambda (R/(I+J_{1}: I))- \lambda ({H_{\frak{m}}^0}(R/I)) \\
=[\lambda (R/(J+J_{1}: I))- \lambda(I/J)] -[\lambda(J:{\frak{m}}^{\infty}/J)-\lambda(I/J)] \\
=\lambda (R/(J+J_{1}: I))- \lambda(J:{\frak{m}}^{\infty}/J)\\
=\lambda (R/(J_{1}:I+J:{\frak{m}}^{\infty} ))+\lambda(J_{1}:I+J:{\frak{m}}^{\infty}/J+J_{1}: I)- \lambda(J:{\frak{m}}^{\infty}/J)\\
=\lambda (R/(J_{1}:I+J:{\frak{m}}^{\infty} ))+\lambda(J:{\frak{m}}^{\infty}/(J+J_{1}: I) \cap J:{\frak{m}}^{\infty} )- \lambda(J:{\frak{m}}^{\infty}/J)\\
=\lambda (R/(J_{1}:I+J:{\frak{m}}^{\infty} ))-\lambda((J+J_{1}: I) \cap J:{\frak{m}}^{\infty}/J)\\
=\lambda (R/(J_{1}:I+J:{\frak{m}}^{\infty} ))-\lambda(J/J)\\
=\lambda (R/(J_{1}:I+J:{\frak{m}}^{\infty} ))\\
=\lambda (R/(J_{1}:I+0:I+I:{\frak{m}}^{\infty} ))\\$\\
where the first equality holds because,
\begin{align*}
\lambda (R/(I+J_{1}: I))&=\lambda (R/(J+J_{1}: I))-\lambda(I+J_{1}: I/J+J_{1}: I)\\
&=\lambda (R/(J+J_{1}: I))-\lambda(I/J+(J_{1}: I \cap I))\\
&=\lambda (R/(J+J_{1}: I))-\lambda(I/J).\\
\end{align*}
Since $\lambda(I/J)< \infty$, we also have
\[\lambda ({H_{\frak{m}}^0}(R/I))=\lambda(I:{\frak{m}}^{\infty}/I)=\lambda(J:{\frak{m}}^{\infty}/I)=\lambda(J:{\frak{m}}^{\infty}/J)-\lambda(I/J)\]
the sixth equality follows by \cite[(9)]{X} that $(J_1:I) \cap (J:{\frak{m}}^{\infty})=J_1$.\\
 \\
 Claim 2:
 $\lambda({ \widetilde  {L_n^{0}} })-\lambda({{L_n^{0}} }) \geq 0$\\
 Proof of Claim 2:
 There is a map
 \[ I^{n+1} :_{I^{n} \cap J_1}  m^{\infty} \longrightarrow I^{n} \cap J_1 /I^{n+1} \cap J_1 +x_1 I^{n-1} \]
 with kernel
 \begin{align*}
[I^{n+1} :_{I^{n} \cap J_1}{\frak{m}}^{\infty}] \cap[ I^{n+1} \cap J_1 +J_1 I^{n-1}]
&=I^{n+1} \cap J_1 +[I^{n+1} :_{I^{n} \cap J_1}{\frak{m}}^{\infty}] \cap J_1 I^{n-1}\\
&=I^{n+1} \cap J_1 +[J_1 I^n :_{I^{n} \cap J_1}{\frak{m}}^{\infty}] \cap J_1 I^{n-1}\\
&=I^{n+1} \cap J_1 +x_1 ( I^n :_{I^{n-1}}{\frak{m}}^{\infty}).
 \end{align*}
 We see that $\lambda({ \widetilde  {L_n^{0}} })-\lambda({{L_n^{0}} }) \geq 0.$
 Then by using the formula for $j_1(I)$, Claim 1 and Claim 2 we have $j_1(I) \geq \lambda (I/J) +\lambda (I^2/JI) + \lambda [R/(J_{1} :_{R} I+(0 :_{R} I+I) :_{R}{\frak{m}}^{\infty})]$, as required.
\end{proof}

 \begin{Theorem}\label{P2}
Assume $R$ is Cohen-Macaulay. Let $I$ be a non $\frak{m}$-primery $R$-ideal which satisfies $\ell(I)=d$, the $G_d$ condition, the $AN_{d-2}^-$ and $\depth(R/I) \geq \min\lbrace 1,\dim R/I \rbrace$. Then for a general minimal reduction $J=(x_1, ... ,  x_d)$ of $I$, one has that $j_1(I) \geq \lambda (I/J) +\lambda (I^2/JI) +\lambda [R/(J_{d-1} :_{R} I+(J_{d-2} :_{R} I+I) :_{R} {\frak{m}}^{\infty})]$.
\end{Theorem}
\begin{proof}
 We prove the theorem by induction on $d$. The case $d=2$ has been proved in Theorem \ref{p1}. Let $d \geq 3$ and assume that the theorem holds for $d-1$.
 Let $ R^{1}=R/J_1 :I$, where $J_1=(x_1)$. Observe that $R^{1}$ is a Cohen-Macaulay ring of dimension $d-1$ and $\ell (IR^{1})=d-1$. Also, $IR^{1}=I+J_1:I/J_1:I$ satisfies $G_{d-1}$ and $AN_{d-3}^-$ (see \cite[Lemma 3.2]{PX} and \cite{U}). By Lemma \ref{l2}, $\depth(R^{1}/IR^{1})\geq\min\lbrace \dim R^{1}/ IR^{1}, 1\rbrace$ and  by \cite[Propositon 2.2]{PX2} \ $J_1(I)=J_1(IR^{1})$ and so we have
 \begin{align*}
J_1(I) =J_1(IR^{1}) & \geq  \lambda(IR^{1}/JR^{1})+  \lambda({I^{2}R^{1}}/JR^{1} IR^{1})\\
& + \lambda [R^{1}/(x_2 , ... , x_{d-1})R^{1} :_{R^{1}} IR^{1}\\
&+(x_2 , ... , x_{d-2})R^{1} :_{R^{1}} IR^{1}+IR^{1}) :_{R^{1}} {\frak{m}}^{\infty})]\\
&= \lambda(I/J)+ \lambda(I^2/(JI+(I^2 \cap J_1:I)) \\
&+\lambda [R/(J_{d-1} :_{R} I+(J_{d-2} :_{R} I+I) :_{R} {\frak{m}}^{\infty})]\\
&= \lambda(I/J)+ \lambda(I^2/JI) +\lambda [R/(J_{d-1} :_{R} I+(J_{d-2} :_{R} I+I) :_{R}{\frak{m}}^{\infty})]
 \end{align*}
that the last equality follow by Lemma \ref{l1}.
 \end{proof}

 \begin{Theorem}\label{P3}
Assume $R$ is Cohen-Macaulay. Let $I$ be a non $\frak{m}$-primary $R$-ideal which satisfies $\ell (I)=d$, the $G_d$ condition, the $AN_{d-2}^-$ and $\depth(R/I) \geq\min\lbrace 1,\dim R/I \rbrace$. Then for a general minimal reduction $J=(x_1, ... ,x_d)$ of $I$, with  $j_1(I) = \lambda (I/J) +\lambda [R/(J_{d-1} :_{R} I+(J_{d-2} :_{R} I+I) :_{R} {\frak{m}}^{\infty})]+1 $, we have $\depth(G(I))\geq d -1$. Moreover, $r_J(I) \leq 2$.
\end{Theorem}

\begin{proof}
By Theorem \ref{P2} we have
$j_1(I) \geq \lambda (I/J)+\lambda (I^2/JI) +\lambda [R/(J_{d-1} :_{R} I+(J_{d-2} :_{R} I+I) :_{R} {\frak{m}}^{\infty})]$
so if $\lambda(I^2/JI)=0$, then  $r(I) \leq 1$ and by  \cite[Theorem 3.8]{PX}, $G(I)$ is Cohen-Macaulay.
 If $\lambda(I^2/JI)=1$, then in this case we use induction on $d$ to prove $\depth G(I)\geq d-1$. If $d=2$, then by \cite[Theorem 4.7]{PX} we have $\depth(G(I))\geq 1$. Let $d\geq 3$ and assume that the theorem holds for $d-1$. Let $\grade(I) \geq 1$ and $ R^{1}=R/J_1 :I$, where $J_1=(x_1)$. Observe that $R^{1}$ is a Cohen-Macaulay ring of dimension $d-1$ and $\ell (IR^{1})=d-1$. Also $IR^{1}=I+J_1:I/J_1:I$ satisfies $G_{d-1}$ and $AN_{d-3}^-$ (see  \cite[Lemma 3.2]{PX}  and \cite{U}). By Lemma \ref{l2}, $\depth(R^{1}/IR^{1})\geq \min\lbrace\dim R^{1}/ IR^{1}, 1\rbrace$ and  by \cite[Lemma 3.2]{PX}  $\lambda(I^{2}R^{1}/JR^{1}IR^{1})=\lambda(I^2/(JI+(I^2 \cap J_1:I)))=\lambda(I^{2} /JI)=1$. Thus we have $\depth(G({IR^{1}}))\geq d -2$. By \cite[Lemma 4.8 and Lemma 4.9]{PX}, one has that $x_{1}^{*}$ is regular on $G(I)$. Since $\depth(G(IR^{1}))\geq d -2$ and $x_{1}^{*}$ is regular on $G(I)$, we have $\depth(G(I))\geq d -1$.
If $\grade(I)=0$, then by Lemma \ref{l2} all assumption still hold for the $R^{0}$ (i.e  $R^{0}$ is Cohen-Macaulay, $\dim R^{0}=\dim R=d$, $IR^{0}$ still satisfies $G_d$ condition, ${AN_{d-2}^{-}}$ , $\ell(IR^{0})=\ell(I)=d$) furthermore $\grade(IR^{0})\geq 1$ and $\depth(G(I))\geq \depth(G({IR^{0}}))$. So if $\grade(I)=0$ then $\depth(G(I))\geq d -1$.
If $j_1(I) = \lambda (I/J) +\lambda (I^2/JI) +\lambda [R/(J_{d-1} :_{R} I+(J_{d-2} :_{R} I+I) :_{R} m^{\infty})]$, then by using formula of $j_{1}(I)$ we have $ \sum\limits_{n=2}^\infty \lambda (I^{n+1}/(JI^{n} +(I^{n+1} \cap J_{1} : I^{\infty})) )=0 $.
If $\depth(G(I))\geq d -1$, then $I^{n+1} \cap  J_i =J_i I^n$ for every $n\geq 0$ and $0 \leq i \leq d-1$. Therefore we have $ \sum\limits_{n = 2}^\infty \lambda (I^{n+1}/JI^{n} )=0$ and so $r_J(I) \leq 2$.
\end{proof}

\begin{Example}
Let $R=k[\kern-0.15em[ x,y,z ]\kern-0.15em]$ be a polynomial ring where $k$ is a field and $\frak{q}=(x^4,xz,yz)$. Set $S=R/\frak{q}$ and let $I=(x, y)$ be an ideal of $S$. Then $S$ is a $1$-dimensional Cohen-Macaulay local ring and $I$ is a Cohen-Macaulay prime ideal that has $\ell(I) =1$, $G_1$ condition and $AN_{-1}^{-}$. By using Macaulay 2 \cite{GS}, we have $\depth(S/I)=1$ and the generalized Hilbert-Samuel polynomial is $P_I(n) =4(n +1) -7$. Hence $j_0(I) =4$ and $j_1(I) =7$. $J=(y)$ is a minimal reduction of $I$ (see \cite [Example 2.2]{MX}) and again by  Macaulay 2 we have $\lambda(I/J)=3$, $\lambda(I^2/JI)=2$ and $\lambda [R/(0 :_{R} I+ I :_{R} {\frak{m}}^{\infty})]=1$. Therefore
\[j_1(I) \geq \lambda (I/J) +\lambda (I^2/JI) +\lambda [R/(0 :_{R} I+ I :_{R} {\frak{m}}^{\infty})].\]
\end{Example}

\begin{Example}
Let $R=k [\kern-0.15em[ x,y,z ]\kern-0.15em]$ be polynomial ring where $k$ is a field and let $\frak{q}=(x^3,xz,yz)$. Set $S=R/\frak{q}$ and let $I=(x, y)$ be an ideal of $S$. Then $S$ is a $1$-dimensional Cohen-Macaulay local ring and $I$ is a Cohen-Macaulay prime ideal that has $\ell(I)=1$, $G_1$ condition and $AN_{d-2}^{-}$ (automatically satisfied since $d =1$). By using Macaulay 2 \cite{GS}, we have $\depth(S/I)=1$ and the generalized Hilbert-Samuel polynomial is $P_I(n) =3(n +1) -4$. Hence $j_0(I) =3$ and $j_1(I) =4$. $J=(y)$ is a minimal reduction of $I$ (see \cite [Example 2.2]{MX}) and again by  Macaulay 2 we have $\lambda(I/J)=2$, $\lambda(I^2/JI)=1$ and $\lambda [R/(0 :_{R} I+ I :_{R} {\frak{m}}^{\infty})]=1$ then we have,
\[j_1(I) = \lambda (I/J) +\lambda (I^2/JI) +\lambda [R/(0 :_{R} I+ I :_{R} {\frak{m}}^{\infty})].\]
\end{Example}


\begin{thebibliography}{}

\bibitem{AM}
R. Achilles and M. Manaresi, {\it Multiplicity for ideals of maximal analytic spread and intersection theory}, J. Math. Kyoto Univ., {\bf 33}(4)(1993), 1029-1046.
\bibitem{AM2}
R. Achilles and M. Manaresi, {\it Multiplicities of a bigraded ring and intersection theory}, Math. Ann., {\bf 309} (1997), 573-591.
\bibitem{C}
C. Ciuperc$\breve{a}$, {\it A numerical characterization of the S2-ification of a Rees algebra}, J. Pure Appl. Algebra, {\bf 178} (2003), 25-48.
\bibitem{FO}
H. Flenner, L. $O^{,}$Carroll and W. Vogel, {\it Joins and Intersections, Monographs in Mathematics}, Springer-Verlag, Berlin, (1999).
\bibitem{GS}
D. R. ~Grayson and M. E. ~Stillman, \emph{Macaulay 2, a software system for research in algebraic
geometry}, Available at http://www.math.uiuc.edu/Macaulay2.
\bibitem{HUC}
S. Huckaba, {\it A d-dimensional extension of a lemma of Huneke's and formulas for the Hilbert coefficients},
Proc. Amer. Math. Soc., {\bf 124} (1996) 1393-1401.
\bibitem{HM}
S. Huckaba and T. Marley, {\it Hilbert coefficients and the depths of associated graded rings}, J. London Math. Soc.,
{\bf 56}(1997), 64-76.
\bibitem{H}
 C. Huneke, {\it Hilbert functions and symbolic powers}, Michigan Math. J., {\bf 34} (1987), 293-318.
 \bibitem{HS}
 C. Huneke and I. Swanson, {\it Integral Closure of Ideals, Rings, and Modules}, Cambridge University Press, Cambridge, (2006).
\bibitem{JU}
M. Johnson and B. Ulrich, {\it Artin-Nagata properties and Cohen-Macaulay associated graded rings}, Compos. Math., {\bf 103}(1996), 7-29.
\bibitem{JK}
B. Johnston and D. Katz, {\it Castelnuovo regularity and graded rings associated to an ideal}, Proc. Am. Math. Soc., {\bf 123}(1995), 727-734.
\bibitem{MX}
P. Manteroa and Y. Xie, {\it Generalized stretched ideals and Sally's conjecture}, J. Pure Appl. Algebra, {\bf 220}(2016), 1157-1177.
\bibitem{NU}
K. Nishida and B. Ulrich, {\it Computing j-multiplicities}, J. Pure Appl. Algebra, {\bf 214} (2010,) 2101-2110.
\bibitem{O}
A. Ooishi, {\it $\Delta$-genera and sectional genera of commutative rings}, Hiroshima Math. J., {\bf 17} (1987), 361-372.
\bibitem{U}
B. Ulrich, {\it Artin-Nagata properties and reductions of ideals}, Contemp. Math., {\bf 159} (1994), 373-400.
\bibitem{PX}
C. Polini and Y. Xie, {\it j-multiplicity and depth of associated graded modules}, J. Algebra, {\bf 372}(2012), 35-55.
\bibitem{PX2}
C. Polini and Y. Xie, {\it Generalized Hilbert functions}, Comm. Algebra, {\bf 42}(2014), 2411-2427.
\bibitem{RV}
M. E. Rossi and G. Valla, {\it Hilbert functions of Filtered modules}, Springer-Verlag, Berlin, (2010).
\bibitem{S}
J. D. Sally, {\it Hilbert coefficients and reduction number 2}, J. Algebraic Geom., {\bf 1}(1992), 325-333.
\bibitem{T}
N. V. Trung, {\it Constructive characterization of the reduction numbers}, Compos. Math., {\bf 137}(2003), 99-113.
\bibitem{X}
Y. Xie, {\it Generalized Hilbert coefficients and Northcott's inequality}, J. Algebra, {\bf 461}(2016), 177-200.
\end{thebibliography}
\end{document}